\def \NN {\mathbb N}
\def \CC {\mathbb C}

\def \RR {\mathbb R}
\def \ZZ {\mathbb Z}

\def \epsilon{\varepsilon}

\def \S  {{\mathcal S}}

\def \si {\sigma}

\renewcommand{\S}{{\mathcal S}}

\documentclass[12pt,reqno]{amsart}

\usepackage{amsmath,amssymb}
\newtheorem{theorem}{Theorem}
\newtheorem{lem}{Lemma}
\newtheorem{cor}{Corollary}
\newtheorem{prop}{Proposition}
\usepackage{hyperref}
\usepackage{booktabs}

\numberwithin{equation}{section}

\begin{document}


\title[]{Forbidden conductors of $L$-functions and continued fractions of particular form} 

\author[]{J.KACZOROWSKI, A.PERELLI \lowercase{and} M.Radziejewski}
\maketitle

\hfill {\it In memory of Professor Andrzej Schinzel}

\medskip
{\bf Abstract.} In this paper we study the forbidden values of the conductor $q$ of the $L$-functions of degree 2 in the extended Selberg class by a novel technique, linking the problem to certain continued fractions and to their weight $w_q$. Our basic result states that if an $L$ function with conductor $q$ exists, then the weight $w_q$ is unique in a suitable sense. From this we deduce several results, both of theoretical and computational nature.



\vskip.5cm
\section{Introduction}

\smallskip
It is expected that the conductor $q$ of an $L$-function from the extended Selberg class $\S^\sharp$, see Section 2 for definitions, can attain only certain special values. For example, it is expected that the $L$-functions in $\S^\sharp$ with degree 2 cannot have conductor $q<1$, and that the $L$-functions in the Selberg class $\S$ always have $q\in\NN$. Both such expectations are far from being proved at present.

\smallskip
In this paper we focus on $L$-functions of degree 2 in $\S^\sharp$ and investigate the problem of the admissible values of their conductor $q$ via an unexpected link with certain continued fractions $c(q,{\bf m})$, which we now define. Let $q>0$ be given. For a vector ${\bf m} =(m_0,\dots,m_k)\in\ZZ^{k+1}$ with some $k\geq0$ we set
\begin{equation}
\label{I1}
c(q,{\bf m}) =  m_{k}+\cfrac{1}{qm_{k-1}+\cfrac{q}{qm_{k-2}+\cfrac{q}{\ddots\,+\cfrac{q}{qm_{0}}}}}.
\end{equation}
Here we assume that {\it all denominators in \eqref{I1} are non-zero}. Such a vector $\bf m$ is called a {\it path} for $q$, or simply a path.

\smallskip
Of course, \eqref{I1} can be translated to the standard continued fraction notation where all numerators are one and indices are in increasing order
\[[a_0,\dotsc,a_k]=a_0+\cfrac{1}{a_1+\cfrac{1}{\ddots\,+\cfrac{1}{a_k}}}
\ \ \ , \ \ \ a_j\in\begin{cases}
\ZZ, & 2\mid j,\\
q\ZZ, & 2\nmid j.
\end{cases}
\]
In this paper we use the notation $c(q,{\bf m})$, as it is a better fit to our transformation formula for $L$-functions.

\smallskip
The fraction $c(q,{\bf m})$ and the  path ${\bf m}$ are called {\it proper} if all $m_j$, $j=0,\dots,k-1$, are non-zero. The proper fractions are those arising naturally in connection with $L$-functions. The integer $k$ is the {\it length} of the path, and clearly $c(q,{\bf m})=m_0$ for a path of length 0. The {\it weight }$w_q({\bf m})$ of $c(q,{\bf m})$ is defined for a path of length $k\geq1$ as
\begin{equation}
\label{I2}
\begin{split}
w_q({\bf m}) &=q^{k/2} \prod_{j=0}^{k-1} |c(q, {\bf m}_j)| \\
&=q^{k/2}\left|m_{k-1}+\cfrac{1}{qm_{k-2}+\cfrac{q}{\ddots\,+\cfrac{q}{qm_{0}}}}\right| \left|m_{k-2}+\cfrac{1}{\ddots\,+\cfrac{q}{qm_{0}}}\right|\cdots\left|m_{0}\right|,
\end{split}
\end{equation}
where ${\bf m}_j=(m_0,\dots,m_j)$ for $0\leq j\leq k$.  If $k=0$ we simply write $w_q({\bf m})=1$. Note that the weight $w_q({\bf m})$ does not depend on the last entry $m_k$, and that always $w_q({\bf m})>0$. Moreover, we say that {\it the weight $w_q$ is unique} if $w_q({\bf m})=w_q({\bf n})$ whenever $c(q,{\bf m})=c(q,{\bf n})$.

\smallskip
The main result of this paper reads as follows.

\begin{theorem}\label{th:1}
If there exists $F\in\S^\sharp$ of degree $2$ and conductor $q$, then the weight $w_q$ is unique.
\end{theorem}

The proof of Theorem \ref{th:1} is based on the properties of certain nonlinear twists of $L$-functions and is given in Section 3.

\smallskip
We shall also prove, see Lemma~\ref{lem:3} in Section 3, that the weight $w_q$ is unique if and only if $w(q,{\bf m})=1$ for all proper fractions of type \eqref{I1} representing 0, i.e. such that $c(q,{\bf m})=0$. A fraction $c(q,{\bf m})$ representing 0 is called a {\it loop}; in such a case, the path ${\bf m}$ is also called a loop. The loop $c(q,(0))$ is the trivial loop.

\smallskip
{\bf Examples.} {\bf 1.} Let $q=2/3$. Then one easily checks that the fraction $c(2/3,(1,-1,-3))$, which has $k=2$, satisfies $c(2/3,(1,-1,-3))=0$. Moreover, using the definition \eqref{I2} we see that $w_{2/3}((1,-1,-3))=1/3 \neq1$. Hence, in view of Theorem \ref{th:1}, there are no functions of degree 2 in $\S^\sharp$ with conductor $q=2/3$.

{\bf 2.} Sometimes loops can be quite long. For $q=7/2$ the sequence
\[
\mathbf{m}=(2,-5,-1,1,-1,1,-1,1,-1,1,2)
\] 
is a loop with $k=10$ and $w_{7/2}(\mathbf{m})=8$. By solving the Diophantine equation 
\[
c(7/2,(m_0,\dotsc,m_k))=0
\] 
for $k<10$ we can see that there are no shorter loops of weight $\neq 1$ for $q=7/2$. As before, we conclude that 
there are no functions of degree 2 in $\S^\sharp$ with conductor $q=7/2$.

{\bf 3.} Choose now $q=\sqrt{3}$ and ${\bf m} = (1,1,1,-1,1)$. A simple computation shows that
\[
c(\sqrt{3},{\bf m})=0 \qquad \text{and} \qquad w_{\sqrt{3}}({\bf m})= \sqrt{3}+2 \neq1,
\]
hence, again, there are no functions of degree 2 in $\S^\sharp$ with conductor $q=\sqrt{3}$.

{\bf 4.} Finally, let $q=2$ and ${\bf m} = (1,-1,1)$. In this case we have
\[
c(2,{\bf m})=0 \qquad \text{and} \qquad w_2({\bf m})=1.
\]
The last equality is not surprising, as it is well known that there exist $L$-functions in $\S^\sharp$ of degree 2 and conductor $q=2$, cf. Lemma \ref{lem:Hecke} with $m=4$. In fact, Theorem \ref{th:1} tells us that 
$w_2({\bf m})=1$  not only for ${\bf m} = (1,-1,1)$ but for {\sl every} loop ${\bf m} =(m_0,\dots,m_k)\in\ZZ^{k+1}$. \qed

\smallskip
As is clear from the above examples, Theorem \ref{th:1} and Lemma~\ref{lem:3} enable to prove non-existence of $L$-functions of degree 2 with a given conductor $q$ by producing a proper loop $c(q,{\bf m})$ with weight $w_q({\bf m})\neq1$. This problem is suitable for computations and, for example, in that way we obtain the following result; see Section 4.

\begin{cor}\label{cor:q-by-n-L}
{\sl There exist no $L$-functions of degree $2$ in $\S^\sharp$ with conductor of the form}
\[
q= \frac{a}{nb} \ \text{with} \ (a,b)=1, \ 2\leq b\leq 300, \ n\geq1
\]
if at least one of the following conditions holds
\begin{itemize}
\item
$b \leq 292$ and $a/b<1$,
\item
$b \leq 150$ and $a/b<3/2$,
\item
$b \leq 100$ and $a/b< 2$,
\item
$b\leq 30$ and $a/b<3$,
\item
$b \leq 9 $ and $a/b < 4$,
\item
$a \leq 25$ and $a/b<4$.
\end{itemize}
\end{cor}

{\bf Remark.} Actually, our computations were performed for $q=a/b$ with $a$, $b$ satisfying one of the conditions specified. The result in Corollary~\ref{cor:q-by-n-L} follows from these computations observing that if $q$ is a forbidden conductor then $q/n$ is also forbidden for every integer $n\geq1$. Suppose indeed that there exists $F\in\S^\sharp$ of degree 2 with conductor $q/n$. Then $FG$ has degree 2 and conductor $q$ for any $G\in\S^\sharp$ of degree 0 and conductor $n$, a contradiction since such functions $G$ actually exist; see \cite{Ka-Pe/1999a}.
Alternatively, Proposition~\ref{prop:q-rescaling} in Section 3 shows the link between properties of fractions $c(q,{\mathbf m})$
and $c(q/n,{\mathbf m})$ without reference to $L$-functions. The second assertion follows from the same computations and properties of loops described in Section~\ref{sec:Computations}.
\qed

\smallskip
Theorem \ref{th:1} justifies a closer study of continued fractions of type (\ref{I1}). There are some natural questions to be asked about them. For instance, we would like to know for which values of $q$ the representation of a real number $a$ in the form $a=c(q, {\bf m})$, with $c(q,{\bf m})$ proper, is unique. For such, our method of detecting forbidden conductors does not work, because, as Lemma~\ref{lem:3} shows, the weight is necessarily unique. So an even more interesting problem is to find all $q$ without the above uniqueness property. Among them, there are $q$'s such that $w(q)$ is not unique, so Theorem \ref{th:1} applies. Hence the basic open question in this direction is to describe the set of such $q$ explicitly.

\begin{theorem}\label{th:trans}
If $q>0$ is transcendental or $q\geq 4$, then
every real number $a$ can be represented as $a=c(q,{\bf m})$ with a proper fraction $c(q,{\bf m})$ at most in one way. In particular, in this case the weight $w_q$ is unique.
\end{theorem}

Let $L(q)$ denote the set of loops for a given $q$.

\begin{cor}\label{cor:odd}
If $q>0$ and $L(q)$ contains a loop of odd length, then $1/q$ is an algebraic integer.
\end{cor}

The proofs of these results do not lie particularly deep. In particular, they do not depend on the theory of $L$-functions. Moreover, it shows that the problem of the uniqueness of the weight $w_q$ is non-trivial for algebraic $q<4$ only. The latter case is far more subtle. In contrast to the proof of Theorem \ref{th:trans}, our proof of the following result heavily depends on $L$-functions, in particular on Theorem \ref{th:1} and the Hecke theory of modular forms for the triangle groups $G(\lambda)$. In passing we remark that Hecke's theory shows the existence of $L$-functions of degree 2 for every conductor $q\geq4$, see Lemma \ref{lem:Hecke}. Thus, although our method cannot detect forbidden conductors among the values $q$ in Theorem \ref{th:trans}, actually there are no forbidden conductors $q\geq4$.

\begin{theorem}\label{th:A} Let $q\in {\mathbb R}$ be a positive algebraic number. The weight $w_q$ is unique in each of the following two cases:

i) $q$ has a Galois conjugate which is greater or equal to $4$; in other words, for a certain $\sigma\in$ {\rm Gal}$(\overline{\mathbb Q}\slash {\mathbb Q})$ we have  $q^{\sigma}\geq 4$;

ii) $q$ is a totally positive algebraic integer.
\end{theorem}

{\bf Remarks.} {\bf 1.} As remarked before, our proof of Theorem \ref{th:A} depends on $L$-functions, but its formulation does not. The case of $q^{\sigma}\geq 4$ also follows from Theorem~\ref{th:trans}.
A natural problem is to give a proof of the second case, independent of the theory of $L$-functions. 

{\bf 2.} From Theorem \ref{th:A}, we know that the weight $w_q$ is unique for the following pair of Galois conjugated algebraic integers $q_{\pm}=(3\pm \sqrt{5})/2$. From Lemma \ref{lem:Hecke} applied with $m=5$ we know that there exists an $L$-function in $\S^{\sharp}_2$ with conductor $q_+$. So, in that case, the uniqueness of $w_{q_+}$ follows from Theorem \ref{th:1}. In contrast, no $L$-function $F\in \S^{\sharp}_2$ with $q_F=q_-$ is known at present, and it is not clear if it exists at all.  Analyzing the proof of Theorem \ref{th:1}, we see that 
the uniqueness of $w_q$ for $q=q_F$, $F\in \S^{\sharp}_2$, follows from consistency conditions imposed by the basic transformation formula, see Lemma \ref{lem:2}, and hence implicitly by the functional equation of $F$. Thus the uniqueness of $w_{q_-}$ can be interpreted as the lack of obstacles for  the existence of $F\in \S^{\sharp}_2$ with $q_F=q_-$.

{\bf 3.} All algebraic integers of the form 
\[ q=4\cos^2(\pi \ell\slash m) \qquad  (m\geq 3, 1\leq \ell <m, (\ell,m)=1)\]
are totally positive. Thus for such $q$'s the weight $w_q$ is unique. In particular it shows that the set of algebraic $q$ for which $w_q$ is unique is dense in the interval $(0,4)$. \qed

\medskip
In the opposite direction we have the following theorem.

\begin{theorem}\label{th:Hecke-by-n}
The weight $w_q$ is not unique for
\[
q=\frac{4}{n}\cos^{2}(\pi\ell/(2k+1)),\qquad k\geq1, \;1\leq\ell<2k+1,\; (\ell,2k+1)=1,\; n\geq2.
\]
In particular, there are no functions of degree $2$ in $\S^{\sharp}$ with such conductors.
\end{theorem}

We conclude with some open problems.

{\bf 1.} Construct an $L$-function $F\in \S_2^{\sharp}$ with conductor $q_F=(3-\sqrt{5})\slash 2$ or show that it does not exist. Show that there exists a real number $q>0$ such that $w(q)$ is unique but there is no $F\in\S^\sharp_2$ with conductor $q$.

{\bf 2.} Show that the set of $q$ for which the weight $w_q$ is not unique is also dense in the interval $(0,4)$.

{\bf 3.} It follows from Theorems \ref{th:trans} and  \ref{th:A} that for every $q>0$ there exists a positive integer $n$ such that $w_{nq}$ is unique. The last question is if for every algebraic $q>0$ there exists a positive integer $n$ such that $w_{q\slash n}$ is not unique. 

\medskip
{\bf Acknowledgements}.  This research was partially supported by the Istituto Nazionale di Alta Matematica, by the MIUR grant PRIN-2017 {\sl ``Geometric, algebraic and analytic methods in arithmetic''} and by grant 2021/41/BST1/00241 {\sl ``Analytic methods in number theory''}  from the National Science Centre, Poland.

\medskip
\section{Definitions and basic requisites}

\smallskip
Throughout the paper we write $s=\si+it$ and $\overline{f}(s)$ for $\overline{f(\overline{s})}$. The extended Selberg class $\S^\sharp$ consists of non identically vanishing Dirichlet series 
\[
F(s)= \sum_{n=1}^\infty \frac{a(n)}{n^s},
\]
absolutely convergent for $\si>1$, such that $(s-1)^mF(s)$ is entire of finite order for some integer $m\geq0$, and satisfying a functional equation of type
\[
F(s) \gamma(s) = \omega \overline{\gamma}(1-s) \overline{F}(1-s),
\]
where $|\omega|=1$ and the $\gamma$-factor
\[
\gamma(s) = Q^s\prod_{j=1}^r\Gamma(\lambda_js+\mu_j) 
\]
has $Q>0$, $r\geq0$, $\lambda_j>0$ and $\Re(\mu_j)\geq0$. The Selberg class $\S$ is, roughly, the subclass of $\S^\sharp$ of the functions having, in addition, an Euler product representation and satisfying the Ramanujan conjecture.  Note that the conjugate function $\overline{F}$ has conjugated coefficients $\overline{a(n)}$, and clearly $\overline{F}\in\S^\sharp$. We refer to the survey papers \cite{Kac/2006},\cite{Ka-Pe/1999b},\cite{Per/2005},\cite{Per/2004},\cite{Per/2010},\cite{Per/2017} for further definitions, examples and the basic theory of the classes $\S^\sharp$ and $\S$.

\smallskip
Degree $d$, conductor $q$ and $\xi$-invariant $\xi_F$ of $F\in\S^\sharp$ are defined as
\[
d=2\sum_{j=1}^r\lambda_j, \qquad q= (2\pi)^dQ^2\prod_{j=1}^r\lambda_j^{2\lambda_j}, \qquad \xi_F = 2\sum_{j=1}^r(\mu_j-1/2):= \eta_F+ id\theta_F
\]
with $\eta_F,\theta_F\in\RR$. In this paper we deal mainly with functions in $\S^\sharp$ of degree $d=2$; the subclass of such functions is denoted by $\S^\sharp_2$.

\smallskip
For $\si>1$ and $F\in\S^\sharp$ with degree 2 and conductor $q$ we consider the nonlinear twist
\begin{equation}
\label{new1}
F(s,\alpha,\beta) = \sum_{n=1}^\infty \frac{a(n)}{n^s} e(-\alpha n-\beta\sqrt{n}),
\end{equation}
where $\alpha,\beta\in\RR$ and $e(x) = e^{2\pi i x}$. Note that, according to our notation above, we have
\[
\overline{F}(s, \alpha,\beta) = \overline{F(\overline{s},\alpha,\beta)} = \sum_{n=1}^\infty \frac{\overline{a(n)}}{n^s} e(\alpha n+\beta\sqrt{n}).
\]
To avoid ambiguities, we also use the following notation when we consider a nonlinear twist of the conjugate function $\overline{F}$
\[
(\overline{F})(s,\alpha,\beta) = \sum_{n=1}^\infty \frac{\overline{a(n)}}{n^s} e(-\alpha n-\beta\sqrt{n}).
\]
Thanks to the periodicity of the complex exponential, for $\alpha\in\ZZ$, the twist in \eqref{new1} reduces to the standard twist.
\[
F(s,\beta) = \sum_{n=1}^\infty \frac{a(n)}{n^s} e(-\beta\sqrt{n}),
\]
and for $m\in\ZZ$, we have 
\begin{equation}
\label{new5}
F(s,\alpha+m,\beta) = F(s,\alpha,\beta). 
\end{equation}
Writing
\[
n_\beta = q \beta^2/4 \quad \text{and} \quad a(n_\beta)=0 \ \text{if} \ n_\beta\not\in \NN,
\]
the spectrum of $F$ is defined as
\begin{equation}
\label{spec}
\text{Spec}(F) := \{\beta>0: a(n_\beta)\neq0\} = \Big\{2\sqrt{m/q}: m\in\NN \ \text{with} \ a(m)\neq 0\Big\}.
\end{equation}
Moreover, for $\ell=0,1,\dots$ we write
\[
s_\ell = \frac{3}{4} -\frac{\ell}{2} \quad \text{and} \quad s^*_\ell = s_\ell -i\theta_F. 
\]

\medskip
\begin{lem}\label{lem:1} Let $\beta\neq0$. Then the standard twist $F(s,\beta)$ is entire if $|\beta|\not\in Spec(F)$, while for $|\beta|\in Spec(F)$ it is meromorphic on $\CC$ with at most simple poles at the points $s_\ell^*$. Moreover, when $|\beta|\in Spec(F)$  the residue of $F(s,\beta)$ at $s=s_0^*$ does not vanish.
\end{lem}

\medskip
We refer to \cite{Ka-Pe/2005}, and \cite{Ka-Pe/2016a} for this and other results on the standard twist. Clearly
\[
\text{Spec$(\overline{F})$ = Spec$(F)$}
\]
and, since $\theta_{\overline{F}}=-\theta_F$, the possible poles of $(\overline{F})(s,\beta)$ are at the points $\overline{s_\ell^*}= s_\ell+i\theta_F$, and $\overline{s_0^*}$ is again a simple pole. 

\medskip
\begin{lem}\label{lem:2}  Let $F\in\S^\sharp$ be of degree $2$ and conductor $q$, and let $\alpha>0$ and $\beta\in\RR$. Then
\begin{equation}
\label{new2}
F(s,\alpha,\beta) = e^{as+b} \ \overline{F}\left(s+2i\theta_F,\frac{1}{q\alpha},-\frac{\beta}{\sqrt{q}\alpha}\right) + h(s)
\end{equation}
with certain $a\in\RR$ and $b\in\CC$, where $h(s)$ is holomorphic for $\si>1/2$.
\end{lem}

\medskip
Since the explicit values of $a$ and $b$ are not specified, this is a less precise form of the Lemma in \cite{Ka-Pe/2017}, in the case where $F$ is suitably normalized. Moreover, Lemma \ref{lem:2} follows by similar but more straightforward arguments in the more general case stated in Lemma \ref{lem:2}, where $\theta_F$ is not necessarily vanishing. We will also need an analogous expression for negative values of the first parameter in $F(s,\alpha,\beta)$, thus for $\alpha>0$, we consider the twist $F(s,-\alpha,\beta)$ and note that.
\[
F(s,-\alpha,\beta) = \overline{(\overline{F})(\overline{s},\alpha,-\beta)}.
\]
Since the conductors of $\overline{F}$ and $F$ are equal, from Lemma \ref{lem:2} we finally obtain that for $\alpha>0$
\begin{equation}
\label{new3}
F(s,-\alpha,\beta) = e^{as+b} \ (\overline{F}) \left(s-2i\theta_F,\frac{1}{q\alpha},\frac{\beta}{\sqrt{q}\alpha}\right) + h(s)
\end{equation}
with certain $a\in\RR$ and $b\in\CC$ and a function $h(s)$ holomorphic for $\si>1/2$.

\smallskip
In the next section, we shall use an argument based on repeated applications of \eqref{new2} and \eqref{new3}. Since what really matters in such an argument is only the value of $1/(q\alpha)$ and $|\beta/(\sqrt{q}\alpha)|$, to simplify notation, we denote by 
\begin{equation}
\label{new4}
\widetilde{F}\left(s\pm 2i\theta_F,\frac{1}{q\alpha},\pm\left|\frac{\beta}{\sqrt{q}\alpha}\right|\right)
\end{equation}
the right hand side of both \eqref{new2} and \eqref{new3}. Clearly, for $\si>1/2$ the function in \eqref{new4} has the same singularities as the functions $\overline{F}$ or $(\overline{F})$ on the right hand side of \eqref{new2} or \eqref{new3}.

\section{Proof of the theorems}

\smallskip
{\bf 3.1. Some properties of fractions and weights.}
We start with some initial properties of the fractions and weights in \eqref{I1} and \eqref{I2}, and we refer to Section 4 for further discussions on this subject. Directly from the definitions, we see that.
\begin{equation}
\label{recurr}
c(q,{\bf m}) = m_k + \frac{1}{qc(q,{\bf m}_{k-1})} \quad \text{and} \quad w_q({\bf m}) = \sqrt{q} |c(q,{\bf m}_{k-1})| w_q({\bf m}_{k-1}).
\end{equation}
Moreover, it is easy to check that for $k\geq2$, we also have
\begin{equation}
\label{skip}
c(q,(m_0,\dots,m_{j-1},0,m_{j+1},\dots,m_k)) = c(q,(m_0,\dots,m_{j-2},m_{j-1}+m_{j+1},m_{j+2},\dots,m_k)).
\end{equation}
Thus, by repeated applications of \eqref{skip} we can transform a path ${\bf m}$ to a proper path ${\bf m^*}$ in such a way that
$c(q,{\bf m})=c(q,{\bf m^*})$. This process is called {\it zero-skipping}, and the notation ${\bf m^*}$ will be used also later on. For example,
\[
c(q,(1,0,1))= c(q,(2))=2 \quad \text{and} \quad c(q,(1,0,-1))=c(q,(0))=0.
\]
We also note that the zero-skipping process preserves the weight, namely $w_q({\bf m}) = w_q({\bf m^*})$ if ${\bf m}$ and ${\bf m^*}$ are as above, since
\[
q^{3/2} \left|m_{j-1}+\frac{1}{0+\frac{q}{qm_{j+1}+\frac{q}{*}}}\right| \left|0+\frac{1}{qm_{j+1}+\frac{q}{*}}\right| \left|m_{j+1}+\frac{1}{*}\right| = \sqrt{q} \left|m_{j-1}+m_{j+1} +\frac{1}{*}\right|.
\]
Further, given two loops ${\bf m} = (m_0,\dots,m_k)$ and ${\bf n} = (n_0,\dots,n_\ell)$ we define the {\it composition} of ${\bf m}$ and ${\bf n}$ as
\[
{\bf mn} = (m_0,\dots,m_k+n_0,\dots,n_\ell).
\]
Note that ${\bf mn}$ is a path and moreover
\begin{equation}
\label{compo1}
c(q,(m_0,\dots,m_k+n_0,\dots,n_j)) = c(q,{\bf n_j}) \quad \text{for} \quad j=0,\dots,\ell.
\end{equation}
Indeed, since ${\bf m}$ is a loop, we have $c(q,(m_0,\dots,m_k+n_0)) = c(q,(n_0))$, hence \eqref{compo1} follows. In particular we have that 
\begin{equation}
\label{compo2}
c(q,{\bf mn})=c(q,{\bf n}).
\end{equation}
As a consequence, we have that
\begin{equation}
\label{twoloops}
\text{if $c(q,{\bf m})$ and $c(q,{\bf n})$ are two loops then $c(q,{\bf mn})$ is also a loop.}
\end{equation}
Clearly, by zero-skipping, we may transform $c(q,{\bf mn})$ to a proper loop with the same weight. Finally, we have that the weight is multiplicative with respect to the composition of loops, namely, if $c(q,{\bf m})$ and $c(q,{\bf n})$ are two loops, then
\begin{equation}
\label{multiplic}
w_q({\bf mn})=w_q({\bf m})w_q({\bf n}).
\end{equation}
Indeed, thanks to \eqref{I2} and \eqref{compo1} we have
\[
\begin{split}
w_q({\bf mn}) &= q^{(k+\ell)/2} \prod_{j=0}^{k-1} |c(q,{\bf m}_j)| \prod_{j=0}^{\ell-1} |c(q,(m_0,\dots,m_k+n_0,\dots,n_j))| \\
&= q^{k/2}  \prod_{j=0}^{k-1} |c(q,{\bf m}_j)| q^{\ell/2}  \prod_{j=0}^{\ell-1} |c(q,{\bf n}_j)| = w_q({\bf m}) w_q({\bf n}).
\end{split}
\]
It can be shown that proper loops form a group under composition (with zero-skipping) and
thus $w_{q}$ is a group homomorphism. The inverse of a loop $\mathbf{m}=(m_{0},\dotsc,m_{k})$
is $(-m_{k},\dotsc,-m_{0})$.

\smallskip
\begin{lem}\label{lem:loop-as-difference}  Let $q>0$.
If two proper fractions satisfy
$c(q,{\bf m})=c(q,{\bf n})$,
then there is a proper loop $\bf u$
such that
${\bf m} = ({\bf u}{\bf n})^*$.
The loop $\bf u$ is non-zero if and only if 
${\bf m}\neq {\bf n}$.
\end{lem}

{\it Proof.} 
Let ${\bf m}=(m_0,\dots,m_k)$ and ${\bf n}=(n_0,\dots,n_\ell)$
be such that
$$
c(q,{\bf m})=c(q,{\bf n}).
$$
If ${\bf m}= {\bf n}$, the assertions are clear. Otherwise,
we note that ${\bf u'} = (m_0,\dots,m_k-n_\ell, -n_{\ell-1},\dotsc ,-n_0)$ is a loop.
Indeed,
$$
c(q,(m_0,\dots,m_{k-1},m_k-n_\ell, -n_{\ell-1},\dotsc ,-n_{j+1}))
=c(q,{\bf n}_j),\qquad j = \ell-1,\dotsc,0,
$$
by induction,
and finally $c(q,(m_0,\dots,m_{k-1},m_k-n_\ell, -n_{\ell-1},\dotsc ,-n_0))=0$.
Let ${\bf u}={\bf u'^*}$, where~$*$ denotes zero-skipping,
and let $j$ be the number of times
\eqref{skip} was applied in this operation.
i.e. the largest integer such that
$$
m_{k-i} = n_{\ell-i},\qquad i = 0,\dotsc,j-1.
$$
We have $j<\min(k+1,\ell+1)$,
otherwise ${\bf u}$ would be a loop with the first or last entry equal to zero, which is impossible.
Hence
\begin{equation*}
\begin{split}
({\bf u}{\bf n})^*
&= ((m_0,\dots,m_{k-j-1},m_{k-j}-n_{\ell-j}, -n_{\ell-j-1},\dotsc ,-n_0)(n_0,\dots,n_\ell))^*
\\
&= ((m_0,\dots,m_{k-j-1},m_{k-j},n_{\ell-j+1},\dots,n_\ell))^*
\\
&= (\bf m)^* = {\bf m}.
\end{split}
\end{equation*}
\qed

\smallskip
\begin{lem}\label{lem:3}  Let $q>0$. The following statements are equivalent

i) the weight $w_q$ is unique;

ii) $w_q({\bf m})=1$ for every proper loop ${\bf m}$;

iii) $w_q({\bf m})=w_q({\bf n})$  for every non-trivial proper loops  ${\bf m}$ and $\bf n$.
\end{lem}

{\it Proof.} The implication i) $\Rightarrow$ ii) follows from the convention that 
$w_q((0))=1$ and $c(q,{\bf m})=0=c(q,(0))$ for every loop ${\bf m}$.
The implication ii) $\Rightarrow$ iii) is trivial. Now we assume iii) and prove ii) first and then i). Let $c(q,{\bf m})$ be a non-trivial proper loop. Consider the composition ${\bf mm}$ and recall that, thanks to \eqref{twoloops}, $c(q,{\bf mm})$ is also a loop. Therefore by \eqref{compo2} and \eqref{multiplic} applied with ${\bf n}={\bf m}$ we get
\[
w_q({\bf m}) = w_q({\bf mm})= w_q({\bf m})^2,
\]
thus $w_q({\bf m})=1$. This implies ii).
%
The equality of fractions $c(q,{\bf m})=c(q,{\bf n})$
implies ${\bf m} = {\bf u}{\bf n}$
for some proper loop $\bf u$,
by Lemma~\ref{lem:loop-as-difference}.
Hence
$$
w_q({\bf m})
= w_q ({\bf u})w_q ({\bf n}) = w_q ({\bf n})
$$
by
\eqref{multiplic} and ii).
This proves i).
\qed

\smallskip
Recall that $L(q)$ denotes the set of loops for a given $q$.

\begin{prop}
\label{prop:q-rescaling} Let $q>0$. Let $\mathbf{m}=(m_{0},\dotsc,m_{k})\in L(q)$ and $r,r'$ be rational numbers with $rr'>0$ such that
\[
m'_{j}=\begin{cases}
rm_{j}, & j\equiv k\pmod{2} \\
r'm_{j}, & j\not\equiv k\pmod{2}
\end{cases}
\]
are all integers. Moreover, let $\mathbf{m'}=(m'_{0},\dotsc,m'_{k})$. Then for $q'=\frac{q}{rr'}$ we have $\mathbf{m'}\in L(q')$
and
\[
w_{q'}(\mathbf{m'})=\begin{cases}
w_{q}(\mathbf{m}), & 2\mid k,\\
\sqrt{\frac{r}{r'}}w_{q}(\mathbf{m}), & 2\nmid k.
\end{cases}
\]
\end{prop}

{\it Proof.} This follows from the relation
\[
c(q',\mathbf{m'}_{j})=\begin{cases}
rc(q,\mathbf{m}_{j}), & j\equiv k\pmod{2},\\
r'c(q,\mathbf{m}_{j}), & j\not\equiv k\pmod{2},
\end{cases}
\]
for $j=0,\dotsc,k$. \qed

\begin{cor}\label{cor:odd-by-n}
If $L(q)$ contains a loop $\mathbf{m}$ of odd length, then the weight $w_{q/n}$ is not unique for any integer $n\geq2$.
\end{cor}

{\it Proof.} Let $q'=\frac{q}{n}$. By Proposition~\ref{prop:q-rescaling} there exist $\mathbf{m'},\mathbf{m''}\in L(q')$ with 
\[
w_{q'}(\mathbf{m'})=\sqrt{n}w_{q}(\mathbf{m})\qquad\text{and}\qquad w_{q'}(\mathbf{m''})=\sqrt{\frac{1}{n}}w_{q}(\mathbf{m}),
\]
so at least one of $w_{q'}(\mathbf{m'})$, $w_{q'}(\mathbf{m''})$ is different from 1.
The corollary follows from Lemma~\ref{lem:3}.
\qed

\medskip
{\bf 3.2. Proof of Theorem \ref{th:1}.}
Theorem \ref{th:1} is an immediate consequence of Lemma~\ref{lem:3} and the following lemma.

\begin{lem}\label{lem:4}   Let $F\in\S^\sharp$ be of degree $2$ and conductor $q$. Then  $w_q({\bf m})=1$ for every proper loop $c(q,{\bf m})$.
\end{lem}

{\it Proof.} Let $\beta\in$ Spec$(F)$; we use the notation in  \eqref{I1},\eqref{I2} and \eqref{new4}. By repeated applications of  \eqref{new5},\eqref{new2} and \eqref{new3} we obtain
\[
\begin{split}
F(s,\beta) &= F(s,m_0,\beta) \\
&= \widetilde{F} \left(s\pm2i\theta_F, \frac{1}{qm_0},\pm \frac{\beta}{\sqrt{q}|m_0|} \right) \\
&= \widetilde{F} \left(s\pm2i\theta_F, m_1+ \frac{1}{qm_0},\pm \frac{\beta}{\sqrt{q}|m_0|} \right) \\
&= \widetilde{F} \left(s + 2n_1 i\theta_F, \frac{1}{qm_1+\frac{q}{qm_0}},\pm \frac{\beta}{q|(m_1+ \frac{1}{qm_0})m_0| } \right) \\
&= \widetilde{F} \left(s + 2n_1 i\theta_F,m_2+ \frac{1}{qm_1+\frac{q}{qm_0}},\pm \frac{\beta}{q|(m_1+ \frac{1}{qm_0})m_0| } \right) \\
&= \dots \\
&= \widetilde{F} \left(s + 2n_{k-1} i\theta_F,c(q,{\bf m}), \pm\frac{\beta}{w_q({\bf m})}  \right)
\end{split}
\]
where ${\bf m}=(m_0,\dots,m_k)$ and the $n_j$'s, $j=1,\dots,k-1$, are certain integers.

\smallskip
If $c(q,{\bf m})$ is a non-trivial proper loop, then the above equation reduces, essentially, to the equality of two standard twists; more precisely, it becomes
\[
F(s,\beta) =  \widetilde{F} \left(s + 2n_{k-1} i\theta_F, \pm\frac{\beta}{w_q({\bf m})}  \right).
\]
Since $\beta\in$ Spec$(F)$, both sides must have a simple pole at $s=s_0^*$ and hence by Lemma \ref{lem:1} we have that $\beta/w_q({\bf m})\in$ Spec$(F)$ as well. Moreover, Lemma \ref{lem:1} implies that $s_0^* + 2n_{k-1}i\theta_F$ must be either $s_0^*$ or $\overline{s_0^*}$. Thus the opposite implication holds as well, namely if $\beta/w_q({\bf m})\in$ Spec$(F)$ then  $\beta\in$ Spec$(F)$. Therefore we have that
\[
\text{$\beta\in$ Spec$(F)$ $\Longleftrightarrow$ $\beta/w_q({\bf m})\in$ Spec$(F)$}.
\]
This, in view of the shape of Spec$(F)$ in \eqref{spec}, implies that $w_q({\bf m}) = 1$, and the lemma follows. \qed

\medskip
{\bf 3.3. Proof of Theorem \ref{th:trans} and its corollary.} 
For a proper path ${\bf m}= (m_0,\ldots,m_k)$ we define the rational function
\begin{equation}
\label{R-1}
R(x, {\bf m}) = m_k +\cfrac{1}{x m_{k-1} +\cfrac{x}{\ddots + \cfrac{x}{xm_0}}}.
\end{equation}
Moreover, we define the polynomials $P_\ell(x, {\bf m})$ and $Q_\ell(x, {\bf m})$, $0\leq \ell\leq k$, inductively as
\begin{equation}
\label{eq:PQ-start}
P_0(x, {\bf m}) \equiv m_0, \qquad Q_0(x,{\bf m})\equiv 1
\end{equation}
 and for $1\leq \ell\leq k$
 \begin{equation}
\label{eq:PQ-step}
 P_\ell (x, {\bf m}) = m_\ell xP_{\ell-1}(x,{\bf m})+ Q_{\ell-1}(x,{\bf m}), \qquad
 Q_\ell(x,{\bf m}) = x P_{\ell-1}(x,{\bf m}).
 \end{equation}
 Then we have
\begin{equation}
\label{R-2}
R(x, {\bf m}) = \frac{P_k(x,{\bf m})}{Q_k(x,{\bf m})}.
 \end{equation}
 By a trivial induction we show that
 \begin{equation}
 \label{eq:degPQ}
 \deg P_\ell=\deg Q_\ell, \qquad0\leq \ell< k.
 \end{equation}

 \smallskip
Now we show that two rational functions of the above type, say $R(x, {\bf m})$ and $R(x, {\bf n})$ with
${\bf m}= (m_0,\ldots,m_k)$ and ${\bf n}= (n_0,\ldots,n_\ell)$, coincide if and only if ${\bf m} = {\bf n}$. Sufficiency is trivial, and so is necessity for $k=\ell=0$. Suppose first that $k,\ell>0$. Then
\begin{equation}
\label{eq:Rmn}
  R(x, {\bf m})=m_k +\frac{1}{xR(x, {\bf m}_{k-1})} \quad \text{and} \quad R(x, {\bf n})=n_l +\frac{1}{xR(x, {\bf n}_{\ell-1})}.
  \end{equation}
By (\ref{eq:degPQ}) we have $R(x, {\bf m})\asymp 1$ and $R(x, {\bf n})\asymp 1$ as $|x|\to\infty$. Thus (\ref{eq:Rmn}) gives \[
m_k=n_\ell + O\left(\frac{1}{|x|}\right)
\]
as $|x|\to\infty$ and hence $m_k=n_\ell$. Again by (\ref{eq:Rmn}) we have $R(x, {\bf m}_{k-1}) = R(x, {\bf n}_{\ell-1})$, therefore by induction we conclude that $R(x, {\bf m}) = R(x, {\bf n})$ implies  ${\bf m}={\bf n}$.
Finally, if $k>0$ and $\ell=0$ (or viceversa) then
\[  
R(x, {\bf m})=m_k +\frac{1}{xR(x, {\bf m}_{k-1})} = n_0,
\]
a contradiction proving our assertion in this case as well.
 
 \smallskip
 After this preparation we can conclude the proof. Suppose that a real number $a$ has two different representations as a proper fraction with transcendental parameter $q$. Then
 \[ 
 a= R(q, {\bf m}) =R(q,{\bf n})
 \]
 for two different proper paths ${\bf m}= (m_0,\ldots,m_k)$ and ${\bf n}= (n_0,\ldots,n_\ell)$. Since the rational functions $R(x, {\bf m})$ and  $R(x,{\bf n})$ are distinct, we deduce that the polynomial
 \[ H(x):= P_k(x, {\bf m}) Q_l(x, {\bf n}) - P_\ell(x, {\bf n}) Q_k(x, {\bf m}) \in {\mathbb Z}[x]\]
 is not identically vanishing and moreover $H(q)=0$. This is impossible if $q$ is transcendental.

\smallskip
Suppose now $q\geq 4$ and suppose ${\bf m}= (m_0,\ldots,m_k)$
is the shortest non-zero proper path such that
$\left|c(q,{\bf m})\right|\leq 1/2$.
We have $\left|m_0\right|\geq 1$, so $k\neq 0$.
Since $\left|c(q,{\bf m}_{k-1})\right|> 1/2$,
we have
$$
\left|\frac{1}{q c(q,{\bf m}_{k-1})}\right| < \frac{1}{2}.
$$
From $m_k = c(q,{\bf m}) - \frac{1}{q c(q,{\bf m}_{k-1})}$
we obtain $\left|m_k\right| < 1$,
contradicting $m_k\neq 0$.
Therefore there is no non-zero proper path ${\bf m}$
such that
$\left|c(q,{\bf m})\right|\leq 1/2$.
In particular, there is no non-zero proper loop.
The assertion follows from Lemma~\ref{lem:loop-as-difference}.\qed

\smallskip
To show the corollary we note that,
by Theorem~\ref{th:trans}, the number $q$ is algebraic. Let
\[
a_{n}x^{n}+\dotsc+a_{1}x+a_{0},
\]
be its minimal polynomial,
where $a_{0},\dotsc,a_{n}$ are integers with $\gcd(a_{0},\dotsc,a_{n})=1$.
For $k=2l+1$ we have $P_{2l+1}(x,\mathbf{m})=x^{l}(1+xH(x))$, where
$H(x)$ is some polynomial with integer coefficients dependent on
$\mathbf{m}$.
If $L(q)$ contains a loop of length $k$, we have
$P_{2l+1}(x,\mathbf{m})=0$.
This implies that the minimal polynomial of $q$ divides $1+xH(x)$.
Hence $\left|a_{0}\right|=1$.
\qed

\medskip
{\bf 3.4. Proof of Theorem \ref{th:A}.}
We need two further lemmas.

\begin{lem}\label{lem:Hecke}
Let $q\in\RR$ be either of the form $q=4\cos^2(\pi/m)$ with $m\geq 3$ or $q\geq 4$. Then there exists an $L$-function in $\S^{\sharp}$ with degree $2$ and conductor $q$.
\end{lem}

{\it Proof.} From the classical Hecke theory we know that there are non-trivial automorphic forms $f$ for the Hecke trangle group $G(\lambda)$ if $\lambda\geq2$ or $\lambda=2\cos(\pi \slash m)$ with integer $m\geq 3$, see e.g. \cite{Hec/1983}.  The corresponding  normalized $L$-function $L_f$ satisfies the functional equation 
\[
\left(\frac{\lambda}{2\pi}\right)^s \Gamma(s+\frac{k-1}{2}) L_f(s)=
\omega \left(\frac{\lambda}{2\pi}\right)^{1-s} \Gamma(1-s+\frac{k-1}{2}) L_f(1-s),
\]
where $\omega =\pm 1$. We cannot claim that $L_f$ belongs to $\S^{\sharp}$ because conjugation of $L_f(1-s)$ is missing in the above functional equation. This however can easily be repaired. Without loss of generality we may assume that $L_f(s)$ has at least one coefficient with a non-zero real part, otherwise we consider $iL_f(s)$. Then $F(s):= L_f(s)+\overline{L_f}(s)$ has real coefficients, satisfies a functional equation of the right type and is not identically zero.  Thus it belongs to $\S^\sharp$, has degree 2 and its conductor equals $\lambda^2$; therefore the lemma follows. \qed

\begin{lem}\label{lem:Kron}
Let $\alpha$ be a totally positive algebraic integer with all conjugates smaller than $4$. Then there exist positive integers $m$ and $\ell$ satisfying $m\geq 3$, $1\leq \ell<m$, $(\ell,m)=1$ such that $\alpha=4\cos^2(\pi \ell\slash m)$. 
\end{lem}

{\it Proof.} Let $\beta:=\sqrt{\alpha}$. Then $\beta$ is a totally real algebraic integer with all Galois conjugates in absolute value smaller than $2$. By the Kronecker theorem, $\beta=2\cos(\pi \ell\slash m)$ for certain positive coprime integers $\ell$ and $m$; see Theorem 2.5 of \cite{Nar/2004}. Thus $\alpha=4\cos^2(\pi \ell\slash m)$. Since $0<\alpha<4$, we have $m\geq 3$. Moreover, by the periodicity of $\cos^2 x$, we can assume that $1\leq \ell <m$, and the lemma follows. \qed

\medskip
Let $q\in \overline{\mathbb Q}$ be a positive algebraic number and suppose that ${\bf m} =(m_0,\ldots, m_k)\in {\mathbb Z}^{k+1}$ is a non-trivial loop for $q$, namely
\[ m_k +\cfrac{1}{q m_{k-1} +\cfrac{q}{\ddots + \cfrac{q}{qm_0}}} =0.\]
Then for every $\sigma\in$ Gal$(\overline{\mathbb Q}\slash {\mathbb Q})$ we have
\[ m_k +\cfrac{1}{q^{\sigma} m_{k-1} +\cfrac{q^{\sigma}}{\ddots + \cfrac{q^{\sigma}}{q^{\sigma}m_0}}}=0\]
so that $\bf m$ is a loop for $q^{\sigma}$ as well. Moreover, it is easy to check that $w_q({\bf m}) = 1$ if and only if $w_{q^{\sigma}}({\bf m})=1$. Hence we conclude that the weights $w_q$ and $w_{q^{\sigma}}$ are simultaneously unique or not. 

\smallskip
We can now complete the proof of Theorem \ref{th:A}. If  $q^{\sigma}\geq 4$ for a certain $\sigma\in$ Gal$(\overline{\mathbb Q}\slash {\mathbb Q})$ then by Lemma \ref{lem:Hecke} there exists an $L$-function in $\S_2^{\sharp}$ with conductor $q^{\sigma}$, and  the weight   $w_{q^{\sigma}}$ is unique by Theorem \ref{th:1}. Consequently $w_q$ is unique as well, thus proving i). To show ii) we assume that $q$ is a totally positive algebraic integer with all conjugates in the interval $(0,4)$. By Lemma \ref{lem:Kron} this means that $q=4\cos^2(\pi \ell\slash m)$ for certain $m\geq 3$ and $1\leq \ell<m$, $(m,\ell)=1$. Let  $\sigma\in$ Gal$(\overline{\mathbb Q}\slash {\mathbb Q})$  be such that $\sigma: \exp(2\pi i \ell\slash m) \mapsto\exp(2\pi i \slash m)$. Then $q^{\sigma}=4\cos^2(\pi\slash m)$. According to Lemma \ref{lem:Hecke} there exists an $L$-function in $\S^{\sharp}_2$ with such a conductor, thus $w_{q^{\sigma}}$ is unique according to Theorem \ref{th:1}. Consequently  $w_q$ is unique as well, and the proof is complete. \qed

\medskip
{\bf 3.5. Proof of Theorem \ref{th:Hecke-by-n}.}
We need the following explicit expression for the polynomials $P_\ell(x,{\bf m})$ defined in the proof of Theorem \ref{th:trans}. The proof of such an expression consists of a straightforward induction, which we omit.

\begin{lem}\label{lem:P_k-form} 
Let $k\geq 1$ and $\mathbf{m}=(m_{0},\dotsc,m_{2k})$. Then
\[
P_{2k-1}(x,\mathbf{m})=x^{k-1}\sum_{j=0}^{k}\left(\sum_{A\in I(2k-1,2j-1)}\prod_{i\in A}m_{i}\right)x^{j}
\]
and
\[
P_{2k}(x,\mathbf{m})=x^{k}\sum_{j=0}^{k}\left(\sum_{A\in I(2k,2j)}\prod_{i\in A}m_{i}\right)x^{j},
\]
where $I(h,j)$ denotes the set of subsets $\left\{ a_{0},\dotsc,a_{j}\right\} $ of $\left\{ 0,\dotsc,h\right\} $
such that $a_{0}<\dotsc<a_{j}$ and $a_{i}\equiv i\pmod{2}$ for every
$i=0,\dotsc,j$.
\end{lem}

Let $k$ and $\ell$ be as in Theorem \ref{th:Hecke-by-n} and
\begin{equation}
\label{qvalue}
q=4\cos^{2}(\pi\ell\slash(2k+1))=\left(e\left(\frac{\ell}{4k+2}\right)+e\left(-\frac{\ell}{4k+2}\right)\right)^{2}.
\end{equation}
By Corollary~\ref{cor:odd-by-n} it suffices to show that $L(q)$ contains a loop of odd length. Let 
\[
\mathbf{m}=(m_{0},\dotsc,m_{2k-1}), \qquad m_{j}=(-1)^{j},\ j=0,\dotsc,2k-1. 
\]
By Lemma~\ref{lem:P_k-form} we have
\begin{align*}
q^{-k+1}P_{2k-1}(q,\mathbf{m}) & =\sum_{j=0}^{k}\left(\sum_{A\in I(2k-1,2j-1)}\prod_{i\in A}(-1)^{i}\right)\left(e\left(\frac{\ell}{4k+2}\right)+e\left(-\frac{\ell}{4k+2}\right)\right)^{2j}\\
 & =\sum_{j=0}^{k}(-1)^{j}\left|I(2k-1,2j-1)\right|\sum_{m=-j}^{j}\binom{2j}{j+m}e\left(\frac{\ell m}{2k+1}\right)\\
 & =\sum_{m=-k}^{k}e\left(\frac{\ell m}{2k+1}\right)\sum_{j=\left|m\right|}^{k}(-1)^{j}\left|I(2k-1,2j-1)\right|\binom{2j}{j+m}.
\end{align*}
The subsets $\left\{ a_{0},\dotsc,a_{2j-1}\right\} \in I(2k-1,2j-1)$
correspond one-to-one to the subsets 
\[
\left\{ b_{0},\dotsc,b_{2j-1}\right\} \subseteq\left\{ 0,\dotsc,k+j-1\right\}
\]
by the mapping $b_{i}=(a_{i}+i)/2$, so $\left|I(2k-1,2j-1)\right|=\binom{k+j}{2j}$.
Hence from the identity
\[
\sum_{i=0}^{m}(-1)^{i}\binom{n-i}{n-2i}\binom{n-2i}{m-i}=1,\qquad n\geq0,\ 0\leq2m\leq n,
\]
which can be shown by induction, it follows that
\begin{align*}
q^{-k+1}P_{2k-1}(q,\mathbf{m}) & =\sum_{m=-k}^{k}e\left(\frac{\ell m}{2k+1}\right)\sum_{j=\left|m\right|}^{k}(-1)^{j}\binom{k+j}{2j}\binom{2j}{j+m}\\
 & =\sum_{m=-k}^{k}e\left(\frac{\ell m}{2k+1}\right)\sum_{j=0}^{k-\left|m\right|}(-1)^{k-j}\binom{2k-j}{2k-2j}\binom{2k-2j}{k-m-j}\\
 & =(-1)^{k}\sum_{m=-k}^{k}e\left(\frac{\ell m}{2k+1}\right)=0.
\end{align*}
Thus for $q$ as in \eqref{qvalue} there exists an integer vector ${\bf m}$ of odd length $2k-1$ with $P_{2k-1}(q,{\bf m})=0$; however, ${\bf m}$ may not be a path for $q$.

\smallskip
Now let $\mathbf{n}=(n_{0},\dotsc,n_{2j-1})\in\mathbf{Z}^{2j}$ be such that $P_{2j-1}(q,\mathbf{n})=0$ with the smallest possible $j$. In view of \eqref{R-1} and \eqref{R-2}, see also Proposition \ref{P2} in the next section, if $\mathbf{n}$ is not a path for $q$ we have $P_{i}(q,\mathbf{n})=0$
for some $i<2j-1$, and $i$ is even by the minimality of $j$. It follows from \eqref{eq:PQ-start} and \eqref{eq:PQ-step} that $Q_{i}(q,\mathbf{n})\neq0$, and moreover that $P_{i+1}(q,\mathbf{n})=Q_{i}(q,\mathbf{n})$. Hence $i+1<2j-1$, so $i< 2j-3$.
Further we have 
\[
Q_{i+1}(q,\mathbf{n})=0, \quad P_{i+2}(q,\mathbf{n})=n_{i+2}qQ_{i}(q,\mathbf{n}) = c P_0(q,\mathbf{n'}), \quad Q_{i+2}(q,\mathbf{n})=qQ_{i}(q,\mathbf{n})=c Q_0(q,\mathbf{n'}),
\]
where $\mathbf{n'}=(n_{i+2},\dotsc,n_{2j-1})$ and $c=qQ_{i}(q,\mathbf{n})\neq 0$. Consequently,
using \eqref{eq:PQ-step} again, we have
\[
P_{i+2+h}(q,\mathbf{n})=c P_{h}(q,\mathbf{n'})\quad \text{for} \  h= 1,\dotsc, 2j-i-3,
\]
and in particular
\[
P_{2j-1}(q,\mathbf{n})=c P_{2j-i-3}(q,\mathbf{n'}).
\]
Hence $P_{2j-i-3}(q,\mathbf{n'})=0$, contrary to the minimality of $j$. Therefore ${\bf n}$ is a path, and also a loop, of odd length; the theorem now follows. \qed

\section{Computations}\label{sec:Computations}

\smallskip
With the aid of machine computations we have been able to find loops of weight $\neq 1$
for rational $q=\frac{a}{b}$, $(a,b)=1$, $0<q<4$, in each of the
following cases:
\begin{itemize}
\item
$a \leq 25$, arbitrary $b$;
\item
$b \leq 292$ and $q<1$;
\item
$b \leq 150$ and $q<3/2$;
\item
$b \leq 100$ and $q< 2$;
\item
$b\leq 30$ and $q<3$;
\item
$b \leq 9 $.
\end{itemize}
An excerpt from the results is shown in Tables~\ref{tab:loops} and~\ref{tab:mod}. Complete
results and the Python scripts needed to reproduce them are available
online at \url{https://maciejr.web.amu.edu.pl/computations/conductors}

\begin{table}
\begin{tabular}{ccc}
\toprule 
$q$ & $\mathbf{m}$ & $w_{q}(\mathbf{m})$\tabularnewline

\midrule
$\frac{1}{2}$ & $(1, -2)$ & $\sqrt{\frac{1}{2}}$ \tabularnewline

\midrule
$\frac{3}{2}$ & $(-1, 1, -2)$ & $\frac{1}{2}$ \tabularnewline

\midrule
$\frac{5}{2}$ & $(-1, 1, -1, 1, 2)$ & $\frac{1}{4}$ \tabularnewline

\midrule
$\frac{7}{2}$ & $(2, 1, -1, 1, -1, 1, -1, 1, -1, -5, 2)$ & $\frac{1}{8}$ \tabularnewline

\midrule
$\frac{1}{3}$ & $(1, -3)$ & $\sqrt{\frac{1}{3}}$ \tabularnewline

\midrule
$\frac{2}{3}$ & $(1, -1, -3)$ & $\frac{1}{3}$ \tabularnewline

\midrule
$\frac{4}{3}$ & $(-1, 1, -3)$ & $\frac{1}{3}$ \tabularnewline

\midrule
$\frac{5}{3}$ & $(-1, 1, -1, -1, -3)$ & $\frac{1}{9}$ \tabularnewline

\midrule
$\frac{7}{3}$ & $(-1, 1, -1, 2, -1, -1, 3)$ & $\frac{1}{27}$ \tabularnewline

\midrule
$\frac{8}{3}$ & $(1, -1, 1, -1, 6)$ & $\frac{1}{9}$ \tabularnewline

\midrule
$\frac{10}{3}$ & $(2, -1, 1, -1, 1, -1, 1, -5, 15)$ & $\frac{1}{81}$ \tabularnewline

\midrule
$\frac{11}{3}$ & $(-1, -1, -1, 1, -1, 1, -1, 1, -1, 1, -1, 1, -30, 1, -8)$ & $\frac{1}{243}$ \tabularnewline

\midrule
$\frac{1}{4}$ & $(1, -4)$ & $\sqrt{\frac{1}{4}}$ \tabularnewline

\midrule
$\frac{3}{4}$ & $(-1, 1, 4)$ & $\frac{1}{4}$ \tabularnewline

\midrule
$\frac{5}{4}$ & $(-1, 1, -4)$ & $\frac{1}{4}$ \tabularnewline

\midrule
$\frac{7}{4}$ & $(1, -1, 1, 2, -2)$ & $\frac{1}{8}$ \tabularnewline

\midrule
$\frac{9}{4}$ & $(-1, 1, -1, 2, 2)$ & $\frac{1}{8}$ \tabularnewline

\midrule
$\frac{11}{4}$ & $(-1, 1, -1, 1, -2, -1, 4)$ & $\frac{1}{32}$ \tabularnewline

\midrule
$\frac{13}{4}$ & $(1, -1, 1, -1, 1, -1, 36)$ & $\frac{1}{64}$ \tabularnewline

\midrule
$\frac{15}{4}$ & $(-2, 1, -1, 1, -1, 1, -1, 1, -1, 1, -1, 1, -6, 11, -1, 8, -1)$ & $\frac{1}{4096}$ \tabularnewline

\midrule
\addlinespace
\addlinespace
\end{tabular}

\caption{\label{tab:loops}Examples of loops for $q=a/b$ with $b\protect\leq 4$.}
\end{table}

\begin{table}
\begin{tabular}{cccccc}
\toprule 
$a$ & $b$ & $N$ & $b''$ & $\bf m$ & $\bf n$\tabularnewline
\midrule

$3$ & $1$ & $3$ & $1$ & $(-1, 0, 2)$ & $(1)$ \tabularnewline
\midrule

$4$ & $1$ & $4$ & $1$ & $(-1, 0, 2)$ & $(1)$ \tabularnewline
\midrule

$5$ & $1$ & $5$ & $1$ & $(-1, 0, 2)$ & $(1)$ \tabularnewline
\midrule

$5$ & $2$ & $10$ & $2$ & $(-2, 0, 3)$ & $(1)$ \tabularnewline
\midrule

$5$ & $3$ & $10$ & none & $(-1, 1, -1, -1, -3)$ & $(0)$ \tabularnewline
\midrule
\addlinespace
\addlinespace
\end{tabular}

\caption{\label{tab:mod}Examples of families of loops of weight $\neq 1$ for $q'=a/b'$
where $b' \equiv \pm b$ (mod $N$), $b\neq b''$, and $N=na$.
For each family we list the data necessary to apply Corollary~\ref{cor:mod-ification}:
$a$, the residue $b$, the modulus $N$, and the paths (denoted $\bf m$ and $\bf n$ in the corollary).
The possible exceptional value of $b'$ in \eqref{eq:residue-with-exception} is denoted as $b''$ here.}
\end{table}

\smallskip
In our computations we make use of some observations that we state
here without complete proofs.

\begin{prop}
\label{P2}
A sequence $\mathbf{m}\in\mathbf{Z}^{k+1}$ is a path for a given
$q$ if and only if $P_{l}(q,\mathbf{m})\neq0$, $0\leq l<k$. In
that case $\mathbf{m}$ is a loop if and only if $P_{k}(q,\mathbf{m})=0$.
Moreover, $w_{q}(\mathbf{m})=\left|q^{-\frac{k}{2}}Q_{k}(q,\mathbf{m})\right|$. 
\end{prop}

\begin{cor}
\label{cor:length-1}
For $q>0$ the set $L(q)$ contains a loop
of length $1$ if and only if $q=\frac{1}{b}$ for some positive integer
$b$. In that case the weight $w_{q}$ is not unique unless $q=1$.
\end{cor}

\begin{proof}
Loops of length $1$ are solutions of $P_{1}(q,(m_{0},m_{1}))=0$,
i.e. $m_{0}m_{1}q+1=0$, so $q$ needs to be of the above form.
In that case $(b,-1)$ is a loop of weight $\sqrt{b}$.
\end{proof}

\begin{cor}
\label{cor:length-2}
For $q>0$ the set $L(q)$ contains a loop
of length $2$ if and only if
$q=\frac{1}{u}+\frac{1}{v}$
for some non-zero integers $u,v$, with $u\neq-v$. In that case the
weight $w_{q}$ is not unique unless $q=1$ or $q=2$. In particular, the
weight is not unique whenever $q=\frac{2}{b}$ for some integer $b\geq3$.
\end{cor}

\begin{proof}
Loops of length $2$ are solutions of $P_{2}(q,(m_{0},m_{1},m_{2}))=m_{0}m_{1}m_{2}q^{2}+(m_{0}+m_{2})q=0$, which implies that
$q$ is of the required form. Conversely, for $q=\frac{1}{u}+\frac{1}{v}$ the sequence $\mathbf{m}=(u,-1,v)$
is always a loop and $w_{q}(\mathbf{m})=\left|\frac{u}{v}\right|$,
which is $\neq1$ unless $v=u$. Suppose $v=u$, so $q=\frac{2}{u}$.
If $2\mid u$, the assertion follows from Corollary~\ref{cor:length-1}.
Otherwise, unless $q=2$, we have $u=2l+1$ for some positive integer
$l$ and there is a loop $\mathbf{m'}=(1,-l,-2l-1)$ of weight $\frac{1}{2l+1}$.
\end{proof}

\begin{prop}
\label{prop:mod-ification}
Let $\mathbf{m}=(m_{0},\dotsc,m_{k})$ be a path
for some $q=\frac{a}{b}$, where $a$ and $b$ are coprime positive integers. Let 
\[
\frac{u_{j}}{v_{j}}=a c(q,\mathbf{m}_{j}),\qquad j=0,\dotsc,k
\]
be reduced fractions, in particular $u_{k}=0$ and $v_{k}=1$ if $\mathbf{m}\in L(q)$. Then
for arbitrary $\varepsilon_{j}=\pm1$, $j=0,\dotsc,k$ and for every
positive $b'$ satisfying $b'\equiv\varepsilon_{j}\varepsilon_{j+1}b\pmod{u_{j}}$,
$j=0,\dotsc,k-1$, the sequence $\mathbf{m'}=(m'_{0},\dotsc,m'_{k})$,
where $m_{0}'=\varepsilon_{0}m_{0}$ and 
\[
m_{j+1}'=\varepsilon_{j+1}m_{j+1}+\frac{\varepsilon_{j+1}b-\varepsilon_{j}b'}{u_{j}}v_{j},\qquad j=0,\dotsc,k-1,
\]
satisfies $c(q',\mathbf{m'}) = c(q,\mathbf{m})$ and $w_{q'}(\mathbf{m'})=\left(\frac{b}{b'}\right)^{k/2}w_{q}(\mathbf{m})$,
where $q'=\frac{a}{b'}$.
\end{prop}

\begin{proof}
It suffices to show that $ac(q',\mathbf{m'}_{j})=\frac{\varepsilon_{j}u_{j}}{v_{j}}$,
$j=0,\dotsc,k$. Indeed, we have $\frac{u_{0}}{v_{0}}=am_{0}=\varepsilon_{0}am_{0}'$
and
\begin{align*}
\frac{u_{j+1}}{v_{j+1}} & =\frac{am_{j+1}u_{j}+abv_{j}}{u_{j}}\\
 & =\pm\frac{am_{j+1}'\varepsilon_{0}u_{j}+ab'v_{j}}{\varepsilon_{0}u_{j}}.
\end{align*}
for $j=0,\dotsc,k-1$.
\end{proof}

\begin{cor}
\label{cor:mod-ification}
Let $\mathbf{m}=(m_{0},\dotsc,m_{k})$ and $\mathbf{n}=(n_{0},\dotsc,n_{l})$
be such that
\[c(q,\mathbf{m})=c(q,\mathbf{n}),
\qquad
\text{for some $q=\frac{a}{b}$, where 
$(a,b)=1$.}
\]
Let 
$
\frac{u_{j}}{v_{j}}=ac(q,\mathbf{m}_{j})
$, $j=0,\dotsc,k-1$,
$\frac{x_{j}}{y_{j}}=ac(q,\mathbf{n}_{j})$, $j=0,\dotsc,l-1$,
be reduced fractions
and let $N$ be a positive integer such that
\[
u_j\mid N,\quad j=0,\dotsc,k-1,
\qquad
\text{and}
\qquad
x_j\mid N,\quad j=0,\dotsc,l-1.
\]
If $k\neq l$, then $w_{q'}$ is non-unique
for every $q'=\frac{a}{b'}$ such that 
\begin{equation}
\label{eq:residue-with-exception}
b'\equiv\pm b\pmod{N},
\qquad
b' \neq \left(\frac{w_{q}(\mathbf{m})}{w_{q}(\mathbf{n})}\right)^{2/(k-l)} b.
\end{equation}
If $k=l$ and $w_{q}(\mathbf{m})\neq w_{q}(\mathbf{n})$,
then $w_{q'}$ is non-unique
for every $q'=\frac{a}{b'}$ such that 
\[b'\equiv\pm b\pmod{N}.\]
\end{cor}

It also follows from Proposition~\ref{prop:q-rescaling} with $r=r'=-1$
and the fact that $w_{q}$ is a group homomorphism that $\mathbf{m}=(m_{0},\dotsc,m_{k})$
is a loop with weight $\neq1$ if and only if $(-m_{0},\dotsc,-m_{k})$
and $(m_{k},\dotsc,m_{0})$ are such loops.

\smallskip
The computations for $q=a/b$ employed several methods, depending on the case
being considered; below is the list of the methods. In the complete table available online, each example of loop
is labelled with the number of the method by which it was obtained.
This number also corresponds to the script number.

\begin{enumerate}
\item The case $q=\frac{1}{b}$ was handled using
Corollary~\ref{cor:length-1}.
\item
For $a = 3,\dotsc,25$
and $b$ relatively prime to $a$ in all possible congruence classes mod $na$, starting with $n=1$,
a search for paths satisfying the assumptions of Corollary~\ref{cor:mod-ification}, with $N=an$, was performed.
(The cases $a=1$ and $a=2$ follow from Corollaries~\ref{cor:length-1} and ~\ref{cor:length-2} respectively).
Examples of paths and congruence classes that we have found are shown in Table~\ref{tab:mod}.
Exceptions (asserted in the corollary)
where noted and later handled by subsequent methods,
with the results stored in the full version of Table~\ref{tab:loops} (online).
Whenever appropriate paths were not found, 
classes mod a higher modulus had to be considered,
either by incrementing $n$, or by splitting the current class mod $an$ to classes mod $ann'$ with the smallest possible $n'$.
The decision to increment $n$ or split the class was based on how many residue classes mod $an$ were already successfully handled.
Covering the next case, $a=26$, with the current algorithm,
would probably require around 2 months of machine time.
\item
For a given $q=a/b$
all possible loops of a given length may be found by solving the Diophantine
equation $P_{k}(q,(m_{0},\dotsc,m_{k}))=0$ in non-zero, integer $m_{j}$.
This was mainly done recursively by 
\begin{itemize}
\item
finding an upper bound $M$ for
$\min\left|m_{j}\right|$,
\item
checking all possible cases of $i=0,\dotsc,k$, $|m_i| = \min\left|m_{j}\right|\leq M$,
\item
substituting possible values 
$|m_i| \leq M$
and then solving each case.
\end{itemize}
For example, for $q=26/23$ and length $6$ the equation to solve is
\begin{multline*}
2^3\cdot13^3 m_0 m_1 m_2 m_3 m_4 m_5 m_6 + 2^2 \cdot 13^2 \cdot 23 (m_0 m_1 m_2 m_3 m_4 + m_0 m_1 m_2 m_3 m_6 \\
 + m_0 m_1 m_2 m_5 m_6 + m_0 m_1 m_4 m_5 m_6 + m_0 m_3 m_4 m_5 m_6 + m_2 m_3 m_4 m_5 m_6)  \\
+ 2\cdot 13\cdot 23^2 (m_0 m_1 m_2 + m_0 m_1 m_4 +  m_0 m_1 m_6 +  m_0 m_3 m_4 +  m_0 m_3 m_6 +  m_0 m_5 m_6\\
 +  m_2 m_3 m_4 +  m_2 m_3 m_6 +  m_2 m_5 m_6 +  m_4 m_5 m_6) 
 + 23^3 (m_0 +  m_2 +  m_4 +  m_6) = 0.
\end{multline*}
This implies $\min|m_j| \leq 2$ which reduces to
$m_j \in \{1,2\}$ for at least one $j$ equal to $0$, $1$, $2$ or $3$.
Each of these 8 substitutions produces an equation in $6$ variables (of degree $6$)
which can be solved by applying the same method recursively.
Ultimately we reach 3384779 subcases involving a quadratic equation in $2$ variables,
for which a dedicated algorithm was used.
Searching for longer loops requires solving much more complex equations,
e.g., those corresponding to length $10$ would not fit on one page.

The computational complexity of this method is hard to estimate.
Roughly, it behaves like $(1+q^{-1})^{k!}$, but it is much less regular.
The unique feature of this method is that it also allows us to prove that loops of a given length do not exist, and thus that a loop of the next possible length (e.g., found with another method) has the smallest possible length.

This method was applied for $3\leq a\leq b\leq300$ and loops of length $2$, $4$
and, sometimes, longer loops.
Results for a given $q$ also showed the existence of loops of weight $\neq 1$ for other $q$,
in accordance with Propositions~\ref{prop:q-rescaling} and~\ref{prop:mod-ification},
allowing us to avoid the direct application of the method for many of the smaller $q$,
where the computation time would be particularly long.

\item\label{method:heuristic}
Where previous methods were unsuccessful, loops
of weight $\neq 1$ were found by examining sequences constructed using the following
simple heuristic:
\begin{itemize}
\item start with $m_{0}=1$ or $m_{0}=b$;
\item given $m_{0},\dotsc,m_{k}$, consider several possible values of $m_{k+1}$
such that 
\[
\left|c(q,(m_{0},\dotsc,m_{k+1}))\right|<C
\]
for some fixed $C>0$;
\item if more than $N$ paths were generated (where $N$ is around $10^{7}$) discard those with largest numerators.
\end{itemize}
\end{enumerate}

\noindent
Direct application of Method~3 allows us to exclude the existence
of loops of a given length and weight $\neq1$ for a given $q$.
This way we were able to check that many of the results in the
full table available online are optimal, in the sense that there
are no shorter loops of weight $\neq 1$ for such $q$.

\bigskip
\bigskip
\noindent
Jerzy Kaczorowski, Faculty of Mathematics and Computer Science, Adam Mickiewicz University, Pozna\'n, 61-614 Pozna\'n, Poland and Institute of Mathematics of the Polish Academy of Sciences, 00-956 Warsaw, Poland. e-mail: \url{kjerzy@amu.edu.pl}

\medskip
\noindent
Alberto Perelli, Dipartimento di Matematica, Universit\`a di Genova, via Dodecaneso 35, 16146 Genova, Italy. e-mail: \url{perelli@dima.unige.it}

\medskip
\noindent
Maciej Radziejewski, Faculty of Mathematics and Computer Science, Adam Mickiewicz University, Pozna\'n, 61-614 Pozna\'n, Poland. e-mail \url{maciejr@amu.edu.pl}


\begin{thebibliography}{100} {\normalsize

\bibitem{Hec/1983} E.Hecke - {\sl Lectures on Dirichlet Series, Modular Functions and Quadratic Forms} - Vanderhoeck $\&$ Ruprecht 1983.

\bibitem{Kac/2006} J.Kaczorowski - {\sl Axiomatic theory of $L$-functions: the Selberg class} - In {\sl Analytic Number Theory}, C.I.M.E. Summer School, Cetraro (Italy) 2002, ed. by A.Perelli and C.Viola, 133--209, Springer L.N. 1891, 2006.

\bibitem{Ka-Pe/1999a} J.Kaczorowski, A.Perelli - {\sl On the structure of the Selberg class, I: $0\leq d\leq 1$} - Acta Math. {\bf 182} (1999), 207--241.

\bibitem{Ka-Pe/1999b} J.Kaczorowski, A.Perelli - {\sl The Selberg class: a survey} - In {\sl Number Theory in  Progress}, Proc. Conf. in Honor of A.Schinzel, ed. by K.Gy\"ory {\sl et al.},  953--992, de Gruyter 1999.

\bibitem{Ka-Pe/2005} J.Kaczorowski, A.Perelli - {\sl On the structure of the Selberg class, VI: non-linear twists} - Acta Arith. {\bf 116} (2005), 315--341.

\bibitem{Ka-Pe/2016a} J.Kaczorowski, A.Perelli - {\sl Twists and resonance of $L$-functions, I} - J. European Math. Soc. {\bf 18} (2016), 1349--1389.

\bibitem{Ka-Pe/2017} J.Kaczorowski, A.Perelli - {\sl A weak converse theorem for degree $2$ $L$-functions with conductor $1$} - Proc. Res. Inst. Math. Sci. Kyoto {\bf 53} (2017), 337--347.

\bibitem{Nar/2004}
W.Narkiewicz - {\sl Elementary and analytic theory of algebraic numbers,} third edition - Springer Verlag 2004. 

\bibitem{Per/2005} A.Perelli - {\sl A survey of the Selberg class of $L$-functions, part I} - Milan J. Math. {\bf 73} (2005), 19--52.

\bibitem{Per/2004} A.Perelli - {\sl A survey of the Selberg class of $L$-functions, part II} - Riv. Mat. Univ. Parma (7) {\bf 3*} (2004), 83--118.

\bibitem{Per/2010} A.Perelli - {\sl Non-linear twists of $L$-functions: a survey} - Milan J. Math. {\bf 78} (2010), 117--134.

\bibitem{Per/2017} A.Perelli - {\sl Converse theorems: from the Riemann zeta function to the Selberg class} - Boll. U.M.I. {\bf 10} (2017), 29--53.


}
\end{thebibliography}
\end{document}